\documentclass[12pt]{amsart} 

\usepackage{amssymb}
\usepackage{latexsym}
\usepackage{amscd}
\usepackage{wasysym}
\usepackage{stmaryrd}
\usepackage{hyperref}
\usepackage[mathscr]{eucal}
\hypersetup{
colorlinks=false,
pdfauthor = {Justin Lynd},
pdftitle = {2-subnormal quadratic offenders and Oliver's p-group conjecture: Justin Lynd},
pdfsubject = {2-subnormal quadratic offenders and Oliver's p-group conjecture},
pdfpagemode = UseNone}

\theoremstyle{plain}
\newtheorem{thm}{Theorem}[section]
\newtheorem{mainthm}{Theorem}

\newtheorem{lem}[thm]{Lemma}
\newtheorem{prop}[thm]{Proposition}

\newtheorem{conj}[thm]{Conjecture}
\newtheorem*{namedtheorem}{\theoremname}
\newcommand{\theoremname}{testing}
\newenvironment{named}[1]{\renewcommand{\theoremname}{#1}
    \begin{namedtheorem}}
    {\end{namedtheorem}}

\theoremstyle{definition}
\newtheorem{definition}[thm]{Definition}

\theoremstyle{remark}

\numberwithin{equation}{section}

\renewcommand{\leq}{\leqslant}
\renewcommand{\geq}{\geqslant}
\renewcommand{\nleq}{\nleqslant}

\newcommand{\norm}{\trianglelefteqslant}

\newcommand{\nin}{\notin}

\newcommand{\set}[1]{\{#1\}}
\newcommand{\setst}[2]{\{\,#1\mid #2 \}}

\newcommand{\inv}[1]{#1^{-1}}

\newcommand{\GL}{\operatorname{GL}}

\newcommand{\Baum}{\operatorname{Baum}}

\newcommand{\X}{\mathfrak{X}}
\renewcommand{\S}{\mathcal{S}}
\newcommand{\A}{\mathcal{A}}
\renewcommand{\hat}{\widehat}

\begin{document}

\title[Oliver's $p$-group conjecture]{2-subnormal quadratic offenders and
Oliver's $p$-group conjecture} 
\author{Justin Lynd} 
\address{Department of Mathematics, The Ohio State University, Columbus, OH 43210}
\email{jlynd@math.ohio-state.edu}
\date{\today}

\begin{abstract} 
Bob Oliver conjectures that if $p$ is an odd prime and $S$ is a finite
$p$-group, then the Oliver subgroup $\X(S)$ contains the Thompson subgroup
$J_e(S)$. A positive resolution of this conjecture would give the existence
and uniqueness of centric linking systems for fusion systems at odd primes. 
Using ideas and work of Glauberman, we prove that if $p \geq 5$, $G$ is a
finite $p$-group, and $V$ is an elementary abelian $p$-group which is an
F-module for $G$, then there exists a quadratic offender which is $2$-subnormal
(normal in its normal closure) in $G$.  We apply this to show that Oliver's
conjecture holds provided the quotient $G = S/\X(S)$ has class at most
$\log_2(p-2) + 1$, or $p \geq 5$ and $G$ is equal to its own Baumann subgroup.
\end{abstract} 
\maketitle

\section{Introduction}
Let $\mathcal{F}$ be a saturated fusion system on a finite $p$-group $S$.  If
$\mathcal{F} = \mathcal{F}_S(G)$ for some finite group $G$, then Broto, Levi,
and Oliver show \cite{BrotoLeviOliver2003b} that one can associate a
classifying space to $\mathcal{F}$ whose completion at $p$ has the homotopy
type of the completion of $BG$ at $p$.  The construction of such a classifying
space is based upon a centric linking system associated to $\mathcal{F}$. But
for arbitrary $\mathcal{F}$, it is an open question\footnote{This question has
now been resolved in the affirmative by Andrew Chermak
\cite{ChermakLocalities}, but Oliver's $p$-group conjecture is still open.} as
to whether an associated centric linking system (and thus a suitable
classifying space for $\mathcal{F}$) exists. See the survey article
\cite{BrotoLeviOliver2004} for more details.

In the course of proving the Martino-Priddy conjecture for odd primes
\cite{Oliver2004}, Oliver defines a certain subgroup $\X(S)$ of $S$, now called
the Oliver subgroup.  He shows that if $p$ is odd and the Thompson subgroup
$J_e(S)$ is contained in $\X(S)$, then the homological obstruction to
the existence and uniqueness of a centric linking system for $\mathcal{F}$
vanishes.

In this paper we investigate Oliver's (purely group-theoretic) conjecture that
this inclusion $J_e(S) \leq \X(S)$ always holds for odd $p$ and prove it for
$p$-groups whose quotient $G = S/\X(S)$ has small nilpotence class. 

\begin{mainthm}\label{main} 
Let $p$ be an odd prime, and suppose $S$ is a finite $p$-group
such that $G = S/\X(S)$ has nilpotence class at most $\log_2(p-2) + 1$. Then
$J_e(S) \leq \X(S)$, so Oliver's conjecture holds for $S$. 
\end{mainthm}

Green, H\'ethelyi, and Mazza show in \cite{GHM2010} that $J_e(S) \leq \X(S)$
if the quotient $G$ is metabelian, of maximal class, or of $p$-rank at most
$p$.  They also settle the case where $G$ has class at most $4$, so the above
result is known for $p \leq 17$.  To our knowledge though, Theorem \ref{main}
provides the first nonconstant lower bound on the class of a counterexample to
Oliver's conjecture.

The \emph{Baumann subgroup} (see \cite{MeierfrankenfeldStellmacherStroth2003}) of a
finite $p$-group $G$ is defined as $\Baum(G) = C_G(\Omega_1Z(J_e(G)))$. In
particular, $\Baum(G) \geq J_e(G)$.  We define in Section \ref{definitions} the
$k$-\emph{th Oliver subgroup} for $3 \leq k \leq p$ in direct analogy with
$\X(S)$, which gives a chain of characteristic, self-centralizing subgroups
$\X_3(S) \leq \X_4(S) \leq \cdots \leq \X_p(S)$ of $S$. The last one $\X_p(S)$
is the original Oliver subgroup $\X(S)$. We prove the following.

\begin{mainthm}\label{baum}
Let $p \geq 5$ and $3 \leq k \leq p$.  Suppose $S$ is a finite $p$-group, and
set $G = S/\X_k(S)$.  If $\Baum(G) = G$, then $J_e(S) \leq \X_k(S)$.
In particular, taking $k = p$, Oliver's conjecture holds for $S$.
\end{mainthm}

Perhaps one reason there has not been a general attack on Oliver's
conjecture to date lies in its resistance to any kind of inductive argument.
The strategy for proving Theorem \ref{baum} gives some hint of how such an
inductive argument may possibly proceed, at least for $p \geq 5$. See Lemma
\ref{normalW} and Section \ref{inductiveapproach} for speculation on this
matter.

In general, we take the point of view of Green, H\'ethelyi, and Lilienthal
in \cite{GHL2008}, where they reformulate Oliver's conjecture in terms of a
statement about representations of the quotient group $G$ over ${\bf F}_p$. In
particular, in a putative counterexample to Oliver's conjecture, $V =
\Omega_1Z(\X(S))$ is an F-module for $G$.

Our methods rely on ideas of Glauberman \cite{Glauberman:priv} and a
modification of Glauberman's generalization of Thompson's Replacement Theorem
in \cite{Glauberman1997}.  For $p \geq 5$ and $V$ an F-module for $G$, the
modification gives rise to the existence of $2$-subnormal quadratic offenders
in $G$, and we use this to prove Theorems \ref{main} and \ref{baum}. Recall
that a subgroup $H$ of $G$ is $2$-\emph{subnormal} if it is normal in a normal
subgroup of $G$, or equivalently, normal in its normal closure in $G$. 

\begin{mainthm}\label{2subnormal}  
Suppose $p \geq 5$.  Let $V$ be an elementary abelian $p$-group
which is an F-module for the $p$-group $G$.  Then there exists a quadratic
offender on $V$ in $G$ which is normal in its normal closure in $G$.  
\end{mainthm}

The organization of the paper is as follows. In Section~\ref{definitions},
we recap definitions and outline the reformulation of Oliver's conjecture
due to Green, et al. Section~\ref{productsubgroups} contains a couple of
elementary lemmas needed for Section~\ref{offenders}, where we prove
Theorem~\ref{2subnormal}. We use Theorem~\ref{2subnormal} in
Section~\ref{classbound} to produce in a counterexample an elementary abelian
subgroup generated by quadratic elements, and we leverage this subgroup to prove
Theorems~\ref{main} and \ref{baum}. In the final Section we record an
observation which could allow for a inductive approach to the conjecture.

\section{Definitions and preliminaries}\label{definitions}

Let $p$ be an odd prime. We outline in this section the
module-theoretic version of Oliver's conjecture from \cite{GHL2008}.

\begin{definition}\label{oliversubgroup} 
Let $3 \leq k \leq p$. The $k$-\emph{th Oliver subgroup} of $S$ is the largest subgroup
$\X_k(S)$ such that there exists a series 
\[ 
1 = Q_0 < Q_1 < \cdots <
Q_{n-1} < Q_n = \X_k(S) 
\] 
with each $Q_i$ normal in $S$ and with the
property that for all $1 \leq i \leq n$, 
\[ 
[\Omega_1(C_S(Q_{i-1})), Q_i;\,k-1] = 1.  
\] 
\end{definition}

Taking $k = p$ in the definition, we get Oliver's original definition, and so
$\X_p(S) = \X(S)$. 
For each $k$, $\X_k(S)$ is a characteristic subgroup of $S$.  Furthermore, as is
clear from the definition, $\X_k(S)$ contains every normal subgroup of $S$ of
nilpotence class at most $k-2$. In particular it contains every normal abelian
subgroup of $S$, and so $C_S(\X_k(S)) = Z(\X_k(S))$ (that is, $\X_k(S)$ is
self-centralizing). A recent preprint from Green, H\'ethelyi and Mazza \cite{GHM2011}
independently calls attention to $\X_3(S)$, which is their $\mathcal{Y}(S)$.

Recall the \emph{Thompson subgroup} is the subgroup $J_e(S)$ of $S$ generated by
the elementary abelian subgroups of maximum order. We will keep the $e$ subscript
throughout, since in Section \ref{offenders}, we will be working with (not
necessarily elementary) abelian subgroups of maximum order. In
\cite[Conjecture \!3.9]{Oliver2004}, Bob Oliver poses the following.

\begin{conj}[{\bf Oliver}]
Let $S$ be a finite $p$-group with $p$ odd. Then $J_e(S) \leq \X(S)$.  
\end{conj}

Perhaps the most effective way of viewing these Oliver subgroups is by way of the
action of the quotient group $G := S/\X_k(S)$ on $V := \Omega_1Z(\X_k(S))$.  If the
$k$-th Oliver subgroup is a proper subgroup of $S$, then for each element
$z \in S$ such that $z\X_k(S)$ is an element of order $p$ in $Z(G)$, the subgroup
$\langle z \rangle \X_k(S)$ is normal in $S$ and properly contains $\X_k(S)$. Thus
by maximality of the Oliver subgroup, 
\[ 
[V, \X_k(S)\langle z \rangle;\,k-1] \neq 1.  
\] 
Since $\X_k(S)$ and $V$ commute, this means that $[V, z;\, k-1] \neq 1$.
In other words, $z \X_k(S)$ acts on $V$ with minimum polynomial of
degree at least $k$.

This motivates the following definition.

\begin{definition} Suppose $G$ is a finite $p$-group, and $V$ is an elementary
abelian $p$-group on which $G$ acts.  Then $V$ is said to be a \emph{PS-module
of degree} $k$ if for $G$ if each $1 \neq z \in \Omega_1Z(G)$ has minimum
polynomial of degree at least $k$. We drop the degree qualifier if $k=p$,
and say $V$ is a PS-module if it is a PS-module of degree $p$.
\end{definition}

Thus, whenever $\X_k(S)$ is a proper subgroup of $S$, $V$ is a PS-module of
degree $k$ for $G$.  Conversely, if $G$ is any finite $p$-group and $V$ is a
PS-module of degree $k$ for $G$, then $\X_k(V \rtimes G) = V$.

Let us examine what happens if Oliver's conjecture fails. If $J_e(S)$ is not
contained in $\X_k(S)$, then there is an elementary abelian subgroup $A \leq S$
of maximum order not contained in $\X_k(S)$. Setting $E = A\X_k(S)/\X_k(S) \leq G$,
this means by maximality of $A$, that $|E||C_V(E)|$ has order at least $|V|$.
This leads us to the next definition and to the reformulation of Oliver's
conjecture due to Green, H\'ethelyi, and Lilienthal in \cite[Theorem \!1.2]{GHL2008}.

\begin{definition} 
Let $G$ be a finite $p$-group and $V$ an elementary abelian $p$-group on
which $G$ acts faithfully.  Then $V$ is an \emph{F-module} for $G$ if there
exists a nontrivial elementary abelian subgroup $E$ of $G$ such that
$|E||C_V(E)| \geq |V|$. In this case, we call $E$ an \emph{offender} and say
that $E$ \emph{offends} on $V$.  
\end{definition}

\begin{thm}
Let $p$ be an odd prime. Oliver's conjecture holds if and only if for every
finite $p$-group $G$, no PS-module for $G$ is an F-module.  
\end{thm}

First appearing in John Thompson's N-group paper as obstructions to
factorizations in $2$-constrained groups, the properties of F-modules are
by now well-known. In particular, as proved by Thompson in his Replacement
Theorem, if there exists an offender $E \leq G$ on $V$, then one may choose $E$
to act \emph{quadratically}, that is, satisfying $[V, E, E] = 1$ but not
centralizing $V$. In particular, the nonidentity elements of $E$ are
\emph{quadratic elements}, that is, they have quadratic minimum polynomial on
$V$. Recall that because of the identity $(x-1)^p = x^p - 1$ in characteristic
$p$, a quadratic element is of order $p$ provided $G$ is faithful on $V$.

When $3 \leq k \leq p$ and $J_e(S)$ is not contained in $\X_k(S)$,
the existence of quadratic offenders in $G = S/\X_k(S)$ on $V =
\Omega_1Z(\X_k(S))$ portends the presence of quadratic elements in $Z(G)$ and
leads the way to a possible contradiction. Indeed, many of the proofs for
special classes of $p$-groups have gone this way. So it is quite possible that
if Oliver's original conjecture holds at all, it holds in the stronger form:
$J_e(S) \leq \X_3(S)$ for all $S$.

In Sections \ref{productsubgroups} and \ref{offenders}, where the focus is
on \emph{abelian} subgroups of $S$ of maximal order, we follow Huppert and
Blackburn's treatment \cite[pp.  19--21]{HuppertBlackburn1982} of Thompson's
Replacement Theorem. Then in the proof of Theorem \ref{baum}, we use inside the
quotient group $G$ a version \cite[Theorem 25.2]{GLS2} of Thompson Replacement
stated in terms of \emph{elementary abelian} subgroups.  Here we state the
following well-known preliminary lemma to Thompson Replacement from
\cite{HuppertBlackburn1982} for use in Section \ref{offenders}.

\begin{lem}\label{hblemma}
Suppose that $S$ is a $p$-group and $A$ is an abelian subgroup of $S$. Let
$v$ be an element of $S$ for which $N = [v, A]$ is abelian, and put
$A^* = NC_A(N)$. Then 
\begin{enumerate}
\item[(a)] $A^*$ is an abelian subgroup and $|A \cap M| \leq
|A^* \cap M|$ for every normal subgroup $M$ of $S$. In particular $|A| \leq |A^*|$. 
\item[(b)] If also $1 = Z_0 \leq Z_1 \leq \cdots \leq Z_n = S$ is a central series of $S$
such that $|A \cap Z_i| = |A^* \cap Z_i|$ for all $i = 1, \dots, n$, then $A^* =
A$.
\end{enumerate}
\end{lem}

\section{Product subgroups}\label{productsubgroups}
Let $V$ be an elementary abelian $p$-group. Suppose $G$ is a $p$-group which
acts faithfully on $V$, and set $S = V \rtimes G$. Let $\pi: S \to G$ denote
the canonical projection.

In this section we make a definition and collect a couple of elementary lemmas
needed for carrying out a modification of Glauberman's generalization of
Thompson Replacement in the next section.

\begin{definition}
Let $B$ be a subgroup of $S$. Define $B_\times$ to be the subgroup $(B\cap
V)\pi(B)$ of $S$. We say $B$ is a \emph{product subgroup} if $B = B_\times$.
\end{definition}

The following Lemma is clear from the definitions, but is included here for
reference.  
\begin{lem}\label{prodbasic}
    Let $B$ be a subgroup of $S$. Then
    \begin{enumerate}
    \item[(a)] $(B_\times)_\times = B_{\times}$.
    \item[(b)] $|B| = |B_\times|$.
    \item[(c)] $B = B_\times$ if and only if $\pi(B) \leq B$.
    \item[(d)] If $B \leq V$ or $V \leq B$, then $B = B_\times$.
    \item[(e)] If $B$ is abelian, then $B_\times$ is abelian.
    \end{enumerate}
\end{lem}

\begin{lem}\label{prodintersect}
Let $A$ and $M$ be subgroups of $S$ and suppose $M \leq V$ or $V \leq M$. Then
$(A \cap M)_\times = A_\times \cap M$.
\end{lem}
\begin{proof}
Since $\pi(A \cap M) \leq \pi(A) \cap \pi(M)$, we always have $(A \cap
M)_\times \leq A_\times \cap M_\times$, and by (d) of Lemma \ref{prodbasic}, the
right-hand side is $A_\times \cap M$, so it suffices to show the reverse
inclusion. 

If $M \leq V$ then $A_\times \cap M = A_\times \cap V \cap M = A \cap V \cap M
= (A \cap M)_\times$, so we may assume $V \leq M$. 

Let $s \in A_\times \cap M$, and write $s = uh$ with $u \in A \cap M
\cap V$ and $h \in \pi(A) \cap \pi(M) \leq M$, the inclusion by (c) of Lemma
\ref{prodbasic}.  For $s$ to be in $(A \cap M)_\times$, it is enough to show
that $h \in \pi(A \cap M)$. But as $h \in \pi(A)$, there exists $v \in V \leq
M$ with $vh \in A$. So $vh \in M$ as well.  Therefore $h \in \pi(A \cap M)$,
completing the proof.
\end{proof}


\section{Offenders}\label{offenders}

In this section, we shall investigate properties of offending subgroups of $G$,
using ideas \cite{Glauberman:priv} and results \cite{Glauberman1997} of
Glauberman, and work toward the proof of Theorem \ref{2subnormal}.

Given an elementary abelian $p$-group $V$ and a $p$-group
$G$ acting faithfully on $V$, we form the semidirect product $S = V \rtimes G$,
as before. 

Denote by $\A(S)$ the set of abelian subgroups of $S$ of maximum order, and
set
\[
\A_\times(S) = \set{A \in \A(S) \mid A = A_\times \mbox{ and } A \nleq V}.
\]
\begin{definition}
Let 
\[ 
\S\!:\, 1 = Z_0 \leq Z_1 \leq \cdots \leq Z_n = S 
\] 
be a central series of $S$.  For abelian subgroups $A$ and $B$ of $S$, we say
that $A \leq_\S B$ if 
\[ 
|A| = |B| \qquad \text{and}\qquad
|A \cap Z_i| \leq |B \cap Z_i| \quad \text{for all} \quad 1 \leq i \leq n.  
\] 
If $A \leq_\S B$ and furthermore $|A \cap Z_i| < |B \cap Z_i|$ for some
$i$, we say that $A <_\S B$. Let $\mathcal{C}$ be a collection of abelian
subgroups of $S$. We say that $A$ is \emph{maximal in} $\mathcal{C}$ \emph{
with respect to} $\S$ if $A \in \mathcal{C}$ and there does not exist $B \in
\mathcal{C}$ with $A <_\S B$. 
\end{definition}

As a consequence of Thompson Replacement, there exists $A \in \A(S)$ such that
for any abelian subgroup $B$ of $S$ which is normalized by $A$, we have $A$ in
turn is normalized by $B$. The Corollary to \cite[Theorem 1]{Glauberman1997}
extends this statement for odd primes to include $B$ of nilpotence class at
most $2$ by choosing $A$ to be maximal in $\A(S)$ with respect to any fixed
central series of $S$.  Furthermore, for $p \geq 5$, Glauberman obtains with
the same choice of $A$ that (\textasteriskcentered) $A$ is normalized by any
subgroup $B$ of $S$ (not necessarily normalized by $A$) for which $A \norm
\langle A^B \rangle$ and $[A, u; 3] = 1$ for every $u \in B$.

To get a statement modulo $V$ about offenders, we must take $A$ not contained
in $V$, and make sure the replacement still lies outside $V$.  To do this, we
modify Glauberman's proof of his Theorem~1 in a couple of ways to obtain a
version of this last result (\textasteriskcentered) for $p \geq 5$.  Firstly,
because of difficulties retaining the commutator condition $[A, u; 3] = 1$
after mapping $u$ into $G$, we restrict the discussion and replacement
process only to elements of $\A_\times(S)$ and require that $B$ be a product
subgroup of $S$.  Secondly, we require that $A$ is maximal with respect to a
central series which contains our distinguished subgroup $V$ as a member.
Since we are just concerned with what happens modulo $V$, this is harmless and
facilitates verification that our replacement subgroup $A^*$ truly is greater
than $A$ in the above ordering.

We begin by showing that $\A_{\times}(S)$ is nonempty and proving a product
subgroup version of a consequence of Thompson Replacement.

\begin{lem}\label{Atimes}
The following hold.
\begin{enumerate}
\item[(a)] If $V$ is an F-module for $G$, then $\A_\times(S)$ is not empty. 
\item[(b)] Let $\S$ be a central series of $S$. Suppose $A$ is maximal in
$\A_\times(S)$ with respect to $\S$.
Then $V$ normalizes $A$, and so $[V, A, A] = 1$.
\end{enumerate}
\end{lem}
\begin{proof}
Suppose $V$ is not an element of $\A(S)$. If $A \in \A(S)$, then in this case
$|A_\times| = |A| > |V|$, and $A_\times$ is a product subgroup by Lemma
\ref{prodbasic}. Thus, $A_\times \in \A_\times(S)$.

Suppose that $V \in \A(S)$. By assumption $S = V \rtimes G$ and $V$ is an
F-module for $G$. Choose a nontrivial elementary abelian subgroup $E \leq G$
such that $|E||C_V(E)| \geq |V|$. Then $A := C_V(E)E \in \A(S)$ by maximality
of $|V|$, and $A$ is a product subgroup. Since $E \neq 1$, $A \in
\A_\times(S)$, proving (a).

Suppose $A$ is maximal in $\A_\times(S)$ with respect to $\S$. Let $E = A \cap
G$, and note that $E \neq 1$ by definition. Set $D = A \cap V = C_V(A)$. 

Suppose $V$ does not normalize $A$ and let $M = N_V(A)$, which is a proper
subgroup of $V$. Then $M$ and $V$ are normal subgroups of $VA$.  Thus, $1 \neq
V/M \norm VA/M$ and so $V/M \cap Z(VA/M) \neq 1$.  Let $v \in V - M$ such that
$v$ maps into $Z(VA/M)$ in the quotient $VA/M$. Then $N := [v, A] \leq M$ is
abelian and normalizes $A$.  Furthermore, $N$ is not contained in $D = C_V(A) <
A$ since $v$ does not normalize $A$.

Thus $D < ND \leq M < V$.  Set $A^* = NC_A(N)$. Since $[D, N] = 1$ and $A = D
\times E$, we get $A^* = NDC_E(N)$ so that $A^*$ is a product subgroup.  Then
as $A^* \cap V = ND > D = A \cap V$, we have that $A^* \neq A$.  By Lemma
\ref{hblemma}(a), $A^*$ is abelian and $|A^*| \geq |A|$.  First, this means
that $A^* \in \A(S)$. Second we get $A^* \cap G \neq 1$ and therefore $A^* \in
\A_{\times}(S)$, since $A^* \cap V = ND$ is proper in $V$ and $|V| \leq |A| =
|A^*|$.  Also, $A \leq_\S A^*$ by Lemma \ref{hblemma}(a).  But $|A \cap Z_i| <
|A^* \cap Z_i|$ for some $i$ by Lemma \ref{hblemma}(b).  This contradicts the
maximality of $A$, and shows that $V$ does indeed normalize $A$. Thus we also
have $[V, A, A] \leq [A, A] = 1$. 

\end{proof}

We now assume for the remainder of this section that $p \geq 5$.

\begin{definition}
Given a subgroup $A$ of $S$, let $\mathcal{B}_A$ be the set of all
$B \leq S$ such that $B = B_\times$, $A \norm \langle A^B \rangle$, and
$[A, u;\,3] = 1$ for all $u \in B$. 
\end{definition}

\begin{thm}\label{glaubcor}
Suppose $\S_V$ is a central series of $S$ which passes through $V$ and $A$ is
maximal in $\A_\times(S)$ with respect to $\S_V$.  Then $A$ is normalized by
every member of $\mathcal{B}_A$.
\end{thm}
\begin{proof}
Suppose not, and choose $B \in \mathcal{B}_A$ such that $B$ does not normalize
$A$. By Lemma \ref{Atimes}(b), $V$ normalizes $A$. As $B = B_\times = (B \cap
V)(B \cap G)$, there exists $b \in B \cap G$ which does not normalize $A$. Set
$T = \langle A, b \rangle$ and $\hat{A} = \langle A^T \rangle$.
Under this setup, Glauberman shows \cite[p.320]{Glauberman1997} (see the last
paragraph of Step 1 there) that
\[
\hat{A} = \langle\, A, A^b, A^{b^2} \,\rangle \quad \text{and} \quad \hat{A} 
\,\,\,\text{has class at most}\,\, 3.
\]
Since $p \geq 5$ and the class of $\hat{A}$ is at most 3, we view $\hat{A}$
as a Lie ring by the Lazard corespondence \cite[Theorem
\!3.4]{Glauberman1997}. Let $\alpha$ be the map on $\hat{A}$ which is
conjugation by $b$, and let $\delta = \log(\alpha)$. By
\cite[pp. 320--323]{Glauberman1997}, we have that $\delta^3 = 0$, 
\[
\delta = (\alpha - 1) - \frac{1}{2}(\alpha - 1)^2\,\,\,\text{is a derivation
of}\,\,\,\hat{A}, 
\]
\[
\hat{A} = A + \delta(A) + \delta^2(A), \mbox{ and }
\]
\[
A \geq Z(\hat{A}).
\]
Set
\[
A_1 = \set{a_1-2\delta(a_2) \mid a_1, a_2 \in A \mbox{ and }
\delta(a_1)-\delta^2(a_2) \in Z(\hat{A})},
\]
and
\[
A^* = A_1 + \delta^2(A).
\]
Then Glauberman shows by way of a Lie ring theoretic result 
\cite[Theorem \!3.2]{Glauberman1997} that as groups, $A^*$ is abelian, $|A| =
|A^*|$, and $A <_{\S} A^*$ for every central series $\S$ of $S$. Set
$A_{\times}^* = (A^*)_\times$. Then by Lemma \ref{prodbasic}(b),
$|A_{\times}^*| = |A^*| = |A|$, and by Lemma \ref{centdelta} below, $A^* \nleq
V$ so $A_{\times}^* \nleq V$. Thus, $A_{\times}^* \in \A_\times(S)$. Since
every member $M$ of $\S_V$ satisfies $M \leq V$ or $V \leq M$, we have by Lemma
\ref{prodintersect} that $A^* \leq_{\S_V} A_{\times}^*$. Therefore, $A <_{\S_V}
A_{\times}^*$. This contradicts the assumption that $A$ is maximal in
$A_\times(S)$ with respect to $\S_V$ and completes the proof.
\end{proof}

\begin{lem}\label{centdelta}
Let $A$, $\delta$, $\hat{A}$ and $A^*$ be as in the proof of \textup{Theorem
\ref{glaubcor}}, with $\hat{A}$ and its subgroups viewed as Lie rings.  Let $D =
A \cap G$, and set $\hat{D} = D + \delta(D) + \delta^2(D)$. Then $A^* \cap
\hat{D} \neq \set{0}$, and therefore $A^* \cap G \neq 1$. 
\end{lem}
\begin{proof}
Note that for the proof, we do not need knowledge that $\hat{D}$ is closed under
the Lie bracket, but only that $\hat{D}$ is a abelian group (as a subspace of
a Lie ring).

Recall that $A^* = A_1 + \delta^2(A)$ where
\[
A_1 = \setst{a_1-2\delta(a_2)}{a_1, a_2 \in A \,\, \text{and} \,\,
\delta(a_1)-\delta^2(a_2) \in Z(\hat{A})}.
\]
As $D = A \cap G$ is nontrivial, we have that $\hat{D}$ is nontrivial.
Since $\delta^2(D)$ is contained in $A^*$, the moment $\delta^2(a) \neq 0$
for some $a \in D$, we are finished. So assume that $\delta^2(D) = 0$. 

Let $a$ be a nonzero element of $D$. Suppose first that $\delta(a) = 0$.
Let $a_2 = 0$. Then $\delta(a) - \delta^2(a_2) = 0 \in Z(\hat{A})$, and
so $a = a - 2\delta(a_2) \in A_1 \leq A^*$, and $a \in \hat{D}$ as well.

Suppose $\delta(a) \neq 0$. If $a \in Z(\hat{A})$, then $\delta(a) \in
Z(\hat{A})$ as well, by \cite[Theorem \!3.2(b)]{Glauberman1997}. Taking $a_2 =
0$ again, we have $\delta(a) - \delta^2(a_2) \in Z(\hat{A})$, giving $a \in A^*
\cap \hat{D}$.  Suppose $a \nin Z(\hat{A})$. By \cite[Theorem
\!3.2(c)]{Glauberman1997}, either $\delta^2(a) \in A^* - A$ or $\delta(a) \in
A^* - A$. The former cannot hold as $\delta^2(a) = 0$, and the latter gives
$\delta(a) \in A^* \cap \hat{D}$.

This shows that $A^* \cap \hat{D} \neq \set{0}$. However,
recall that $\delta$ is a polynomial in $\alpha$, which is an automorphism
induced by conjugation by an element of $G$. It follows that $\hat{D} \leq G$,
and $A^* \cap G$ is nontrivial as claimed.
\end{proof}

The following is basically the content of \cite[Theorem \!4.1]{Glauberman1997}.
However, since we have restricted the collections from which $A$ and $B$ are
chosen in proving Theorem \ref{glaubcor}, the statement and argument have to be
modified. We present a complete proof in our case for the convenience of the
reader.
\begin{thm}\label{normeachother}
Suppose $A$ and $B$ are abelian product subgroups of $S$, both normalized by
$V$.  Assume that $A$ is normalized by every member of $\mathcal{B}_A$ and
$B$ is normalized by every member of $\mathcal{B}_B$. Then $A$ and $B$
normalize each other. 
\end{thm}
\begin{proof}
We proceed by induction on the order of $G$. If $G = 1$, then $S = V$, and
the statement is trivial. Suppose that $G > 1$. 

Assume first that $A \cap G = G$. Then $G$ is abelian and so the conjugation
action of $S$ is trivial modulo $V$. Hence $\langle A^B \rangle \leq VA$, so
$A$ is normal in $\langle A^B \rangle$ because $V$ normalizes $A$. As $[S, S]
\leq V$ and $V$ normalizes $B$, we also have $[A, u, u, u] \leq [V, u, u] \leq
[B, u] = 1$ for all $u \in B$. So $B \in \mathcal{B}_A$. Because $G$ is abelian
we still have $\langle B^A \rangle \leq VB$, and the same argument applies to
give $A \in \mathcal{B}_B$.  Thus $A$ and $B$ normalize each other. By
symmetry, the same conclusion holds in case $B \cap G = G$. 

So we may assume that $A \cap G$ and $B \cap G$ are proper in $G$. Let $M_0$
be a maximal subgroup of $G$ containing $A \cap G$, and set $M = VM_0$. Then $M
\norm S$, so
\[
A \leq \langle A^B \rangle \leq \langle M^S \rangle = M.
\]
If $x \in B$, then $A^x = A^{\pi(x)}$, since $V$ normalizes $A$. Hence
$A^x$ is a product subgroup. Furthermore, it is easy to see that
$A^x$ is normalized by every member of $\mathcal{B}_{A^x} = (\mathcal{B}_A)^x$.
By the inductive hypothesis, $A^x$ normalizes $A$. 
It follows that $A \norm \langle A^B \rangle$, and
\[
[B, A;\,3] \leq [\langle A^B \rangle, A, A] \leq [A, A] = 1.
\]
By symmetry $[A, B;\,3] = 1$. This means that $B \in \mathcal{B}_A$. Hence,
$B$ normalizes $A$ by hypothesis. By symmetry, $A$ normalizes $B$.
\end{proof}

Now we are in a position to prove Theorem \ref{2subnormal}.

\begin{named}{Theorem \ref{2subnormal}}
Suppose $p \geq 5$.  Let $V$ be an elementary abelian $p$-group
which is an F-module for the $p$-group $G$.  Then there exists a quadratic
offender on $V$ in $G$ which is normal in its normal closure in $G$.  
\end{named}
\begin{proof}
Let $S = V \rtimes G$. By Lemma \ref{Atimes}(a), $\A_{\times}(S)$ is nonempty.
Let $\S$ be a central series which passes through $V$ and let $A$ be 
an element of $\A_\times(S)$ maximal in $\A_\times(S)$ with respect to $\S$.
Set $E := A \cap G \neq 1$. 
As $A \in \A(S)$, we have $|V| \leq |A|$. Since $A$ is a product
subgroup, $|V| \leq |A| = |A \cap V||A \cap G| = |C_V(E)||E|$. So $E$
offends on $V$.
By Lemma \ref{Atimes}(b), $V$ normalizes $A$, and $E$ acts quadratically on
$V$. Let $B$ be a conjugate of $A$. Then $B$ is a product subgroup because
$V$ normalizes $A$. Thus, $B$ is maximal in $\A_\times(S)$ with respect to
$\S$. By Theorem \ref{glaubcor}, $B$ is normalized by each element in
$\mathcal{B}_B$. By Theorem \ref{normeachother}, $A$ is normalized by its
$S$-conjugates. Therefore $E$ is normalized by its $G$-conjugates. This is
another way of saying that $E \norm \langle E^G \rangle$.
\end{proof}

\section{Oliver's conjecture}\label{classbound}

In this section we leverage the existence of a $2$-subnormal quadratic offender
guaranteed in Theorem~\ref{2subnormal} to prove Theorems~\ref{main} and \ref{baum}.
First we state the key lemma of Green, H\'ethelyi, and Lilienthal in 
\cite[Lemma 4.1]{GHL2008}.

\begin{lem}\label{greenlemma}
Suppose $V$ is an elementary abelian $p$-group and $G$ is a $p$-group acting
faithfully on $V$. Let $a$ be an element of $G$ acting quadratically on $V$.
Let $x \in G$ and suppose $c := [a, x] \neq 1$ centralizes both $a$ and $x$.
Then $c$ acts quadratically on $V$.
\end{lem}

The previous Lemma allows one, under certain conditions, to begin with a
quadratic element and descend the central series of $G$, finding quadratics
further and further down. This process stops when at some point, and having
some quadratic element $a$, one can find no element $x \in G$ for which
$[a,x] \neq 1$ centralizes $a$, and $[a, x, x] = 1$. The proof of the next
Lemma uses the existence of a $2$-subnormal offender to produce normal abelian
subgroups of $G$ which act as a container for descent, and identifies a normal
abelian subgroup $W$ at the bottom when the descent process terminates.

\begin{lem}\label{normalW}
Suppose $p \geq 5$ and $3 \leq k \leq p$.  Suppose $S$ is a finite $p$-group
such that $J_e(S)$ is not contained in $\X_k(S)$. Set $G = S/\X_k(S)$,
which acts faithfully on $V = \Omega_1Z(\X_k(S))$. Then
there exists a normal elementary abelian subgroup $W = \langle a^G \rangle$ of
$G$, generated by the $G$-conjugates of a quadratic element, such that for each $x
\in G$, either $[W, x] = 1$ or $[W, x, x] \neq 1$. 
\end{lem}
\begin{proof}
Note $J_e(V \rtimes G)$ is not contained in $\X_k(V
\rtimes G) = V$, and so we assume that $S = V \rtimes G$. Furthermore,
$V$ is both a PS-module of degree $k$ and an F-module for $G$. 
Choose by Theorem \ref{2subnormal} a quadratic offender $E \leq G$ which
is normalized by its conjugates in $G$. Let $N$ be the normal closure
$\langle E^G \rangle$ of $E$ in $G$, so that $E \norm N$. 
Then $E \cap \Omega_1Z(N) \neq 1$, and hence there are quadratic elements
in $\Omega_1Z(N)$. For each $i \geq 0$, set $W_i = [\Omega_1Z(N), G;\,i]$.
Let $j$ be the greatest integer such that there exists a quadratic element in
$W_j$. Choose such a quadratic element $a \in W_j$. Note that $W_j$ is not in
the center of $G$ as $Z(G)$ has no quadratic elements. Set $W = \langle a^G
\rangle$.  Then $W$ is a normal elementary abelian subgroup of $G$ generated by
quadratic elements.

Let $x \in G$. Suppose that $x$ does not centralize $W$, 
but that $[W, x, x] = 1$. Then there exists a $G$-conjugate $b$ of $a$
such that $[b, x] \neq 1$, but $[b, x, x] = 1$. Setting $c = [b, x]$, 
we have that $c$ centralizes both $b$ and $x$. Since $b$ is quadratic,
$c$ is as well by Lemma \ref{greenlemma}. Thus $c \in W_{j+1}$ and
quadratic. This contradicts the maximality of $j$. Therefore,
for every $x \in G$, we have either $[W, x] = 1$ or $[W, x, x] \neq 1$.
\end{proof}

We have recently learned that Green, H\'ethelyi, and Mazza have independently
obtained a similar subgroup (generated by \emph{last quadratics} in the
terminology of their paper) as the $W$ in Lemma \ref{normalW} under only the
assumption that $G$ has quadratic elements on $V$. See Section~8 of their paper
\cite{GHM2011}.

\begin{named}{Theorem \ref{baum}}
Suppose $p \geq 5$ and $3 \leq k \leq p$. Let $S$ be a finite $p$-group,
and set $G = S/\X_k(S)$.  If $\Baum(G) = G$, then $J_e(S) \leq \X_k(S)$. 
In particular, Oliver's conjecture holds in this case.
\end{named}
\begin{proof}
Suppose not and as in the proof of Lemma \ref{normalW}, let $S = V \rtimes G$
be a counterexample.  Let $W = \langle a^G \rangle$ be as in Lemma
\ref{normalW}. Let $\A_e(G)$ denote the set of elementary abelian subgroups
of $G$ of maximum order.

By Thompson Replacement \cite[Theorem 25.2]{GLS2}, there exists $A^* \in \A_e(G)$
acting quadratically on $W$ or else $[W, A] = 1$ for all $A \in \A_e(G)$.
The first possibility is ruled out by Lemma \ref{normalW}. 
It follows that $[W, A] = 1$, whence $W \leq A$ for all $A \in \A_e(G)$.
So $W \leq \Omega_1Z(J_e(G)) = \Omega_1Z(G)$ as $\Baum(G) = G$. Since $W$ contains
quadratic elements and $V$ is a PS-module of degree $k \geq 3$ for $G$, this is
a contradiction.  
\end{proof}

We need the following basic lemma for the proof of Theorem \ref{main}.
\begin{lem}\label{commutingquadratic}
Suppose $G$ is a $p$-subgroup of $\GL(n, p)$. Assume that $a$ and $b$ are
commuting elements of $G$ whose minimum polynomials on $V$ have degree at most
$k$ and $l$, respectively.  Then $ab$ has minimum polynomial of degree at most
$k + l - 1$.  In particular, if $a_1, \dots, a_k$ are commuting quadratic
elements, then $a_1\cdots a_k$ has minimum polynomial of degree at most $k+1$.
\end{lem}
\begin{proof}
We have $ab - 1 = (a-1) + (b-1) + (a-1)(b-1)$.
Raise the right hand side to the $k + l - 1$st power, and distribute.
Since $a$ and $b$ commute, each term is of the form $(a-1)^i(b-1)^j$
with $i + j \geq k + l - 1$. As $(a - 1)^k = 0$ and $(b - 1)^l = 0$, one
of $(a - 1)^i$ or $(b - 1)^j$ must be zero in the respective term. Therefore,
$(ab - 1)^{k+l-1} = 0$.
The last statement follows by induction on $k$. 
\end{proof}

\begin{named}{Theorem \ref{main}}
Let $p$ be an odd prime, and suppose $S$ is a finite $p$-group
such that $S/\X(S)$ has nilpotence class at most $\log_2(p-2) + 1$. Then
$J_e(S) \leq \X(S)$, so Oliver's conjecture holds for $S$. 
\end{named}
\begin{proof}
The conjecture has been shown in \cite[Theorem \!5.2]{GHM2010} to hold
when $S/\X(S)$ has nilpotence class at most $4$, so the result is known
for small primes. Therefore, we may assume that $p \geq 5$.

Suppose not and choose a counterexample $S = V \rtimes G$ with $V$ both
a PS-module and an F-module for $G$. Let $W = \langle a^G \rangle$
be the normal elementary abelian subgroup of $G$ guaranteed by Lemma
\ref{normalW}.

As $W \norm G$, there exists an integer $k > 0$ such that $1 \neq [W, G;\,k] \leq
\Omega_1Z(G)$. Hence, there exists a $G$-conjugate $b$ of $a$, and elements
$x_1$, \dots, $x_k$ of $G$, such that
\[
1 \neq [b, x_1, \dots, x_k] \in \Omega_1Z(G).
\]
Applying ($k$ times) the standard identity $[x, y] = \inv{x}x^y$ one sees
that
\[
z := [b, x_1, \dots, x_k] = \prod_{j=0}^k\,\,\prod_{(i_1, \dots, i_j) \in
X_j}(b^{(-1)^{j+1}})^{x_{i_1}\cdots x_{i_j}},
\]
where $X_j$ is the set of strictly increasing sequences of length $j$ from
$\{1, \dots, k\}$. It follows that $z$ is a product of $\sum_{j=0}^k
\binom{k}{j} = 2^k$ commuting conjugates of $b$ and $\inv{b}$. By
Lemma \ref{commutingquadratic}, the degree of the minimum polynomial of
$z$ on $V$ is at most $2^k + 1$. On the other hand, the degree of the minimum
polynomial of $z$ is $p$, because $V$ is a PS-module for $G$. Thus $p < 2^k +
2$, and hence $\log_2(p-2) < k$. By definition of $k$, we have that the class
of $G$ is at least $k + 1$. Therefore the class of $G$ is strictly larger
than $\log_2(p-2) + 1$, and this completes the proof.
\end{proof}

\section{Speculation regarding an inductive approach}\label{inductiveapproach}

Lemma \ref{normalW} applied with $k=3$ gives some hint as to how an inductive
approach to the conjecture might proceed at least for $p \geq 5$. (In
view of the subgroup generated by last quadratics in \cite[Section~8]{GHM2011},
the approach applies for all odd $p$.) We record here an observation that forms
the main motivation for this idea.

\begin{prop}\label{X_3(G)proper}
Suppose that $p \geq 5$ and $S$ is a $p$-group such that $J_e(S)$ is not
contained in $\X_3(S)$. Let $G = S/\X_3(S)$. Then $\X_3(G)$ is a proper
subgroup of $G$. 
\end{prop}
\begin{proof}
As $J_e(S)$ is not contained in $\X_3(S)$, $\X_3(S)$ is a proper subgroup of
$S$. Let $W$ be as in Lemma~\ref{normalW}. Thus $W$ is a normal
elementary abelian subgroup of $G$ generated by quadratic elements.
Set $X := C_G(W)$, and note that $X$ is a proper normal subgroup of $G$.
We will show that $\X_3(G) \leq X$, and this will settle the claim.  Suppose to
the contrary that $\X_3(G)$ is not contained in $X$. Let $1 = Q_0 < Q_1 <
\cdots < Q_n = \X_3(G)$ be a series of normal subgroups of $G$ such that 
\[ 
[\Omega_1(C_G(Q_{i-1})), \,Q_{i}, \,Q_i] = 1
\]
for each $i$. Let $j$ be the greatest integer such that $Q_{j-1} \leq X$. Let
$x \in Q_{j} - X$. Then as $W \leq \Omega_1Z(X) \leq \Omega_1(C_G(Q_{j-1}))$,
we have
\[
[W, \,x, \,x] \leq [\Omega_1(C_G(Q_{j-1})), \,x, \,x] = 1,
\]
a contradiction to Lemma \ref{normalW} because $x \nin X$.  Therefore, $\X_3(G)
\leq X$, and $\X_3(G)$ is a proper subgroup of $G$. 
\end{proof}

For an inductive approach, we would choose $S$ to be minimal subject to the
condition that $J_e(S)$ is not contained in $\X_3(S)$. We would then like to
show that $J_e(G)$ is not contained in $\X_3(G)$ to contradict the minimality
of $S$. By the proof of Proposition \ref{X_3(G)proper}, we have that $\X_3(S)
\leq C_G(W)$. However, the content of the proof of Theorem \ref{baum} is
that $J_e(G)$ is also contained in $C_G(W)$. In fact, one can see this
explicitly in the example of Green, H\'ethelyi, and Lilienthal in Section~5 of
their first paper \cite{GHL2008}. In this case $G \cong C_3 \wr C_3$ acting
on an $8$-dimensional module $V$, and we have that $X = W = J_e(G) = \X_3(G)$
is the base subgroup of $G$. Note that $G$ is not a counterexample to Oliver's
conjecture, as can be seen from subsequent work (\!\!\cite[Theorem~1.1]{GHM2010} or
\cite{GHM2011}). But this example shows that taken alone, the existence of a
normal elementary abelian subgroup of $G$ generated by quadratic elements is
not enough to carry out this inductive argument.  

\section*{Acknowledgements}

I would like to thank Ian Leary for introducing me to Oliver's conjecture over
lunch some time ago. George Glauberman provided me with the starting point for
these ideas, and I am grateful to him for his encouragement while writing up
the details. Theorem \ref{main} was proven on the beautiful Isle of Skye at the
Conference on Algebraic Topology, Group Theory and Representation Theory in
June 2009; thanks to the organizers for providing such a great setting for
mathematics.  I am indebted to the referee for a very careful reading of this
paper and many helpful comments on style and substance which greatly improved
the exposition.  Finally, I would like to thank David Green, L\'aszl\'o
H\'ethelyi, and Nadia Mazza for keeping me up to date on their work, and my
thesis advisor Ronald Solomon for his constant encouragement and support.

\bibliographystyle{amsalpha}{}
\bibliography{/home/jlynd/math/research/mybib}

\end{document}